\documentclass[reqno]{amsart}
\usepackage{hyperref,amsmath,mathrsfs}

\DeclareMathOperator*{\esssup}{ess\,sup}
\DeclareMathOperator*{\essinf}{ess\,inf}

\begin{document}
\title[Hemivariational Inequalities on Graphs]
{Hemivariational Inequalities on Graphs}

\author[N. Ait Oussaid]
{Nouhayla Ait Oussaid}

\author[K. Akhlil]
{Khalid Akhlil}

\author[S. Ben Aadi]
{Sultana Ben Aadi}

\author[M. El Ouali]{Mourad El Ouali}

\author[A. Srivastav]{Anand Srivastav}

\address{\newline Nouhayla Ait Oussaid, Khalid Akhlil, Sultana Ben Aadi and Mourad El Ouali\newline
Department of Mathematics and Management \newline
Polydisciplinary Faculty of Ouarzazate\newline
Ibn Zohr University\newline
Agadir, Morocco.}
\address{\newline Anand Srivastav\newline
Research Group Discrete Optimization\newline
Christian-Albrechts-Universität zu Kiel\newline
Kiel, Germany.}

\email{~\newline  nouhayla.aitoussaid@gmail.com\newline
k.akhlil@uiz.ac.ma\newline 
 sultana.benaadi@edu.uiz.ac.ma\newline
 mauros1608@gmail.com\newline
 srivastav@math.uni-kiel.de
 }


\date{\today}
\subjclass[2000]{49J40;49J52;49J53;05C63;35J60;35K55;49A70
}
\keywords{Locally finite graphs; Nonconvex sum functionals; Hemivariational inequality; Clarke's subdifferential}

\begin{abstract}
In this paper, a new class of hemivariational inequalities is introduced. It concerns Laplace operator on locally finite graphs together with multivalued nonmonotone nonlinearities expressed in terms of Clarke's subdifferential. First of all, we state and prove some results on the subdifferentiability of nonconvex functionals defined on graphs. Thereafter, an elliptic hemivariational inequality on locally finite graphs is considered and the existence and uniqueness  of its weak solutions are proved by means of the well-known surjectivity result for pseudomonotone mappings. In the end of this paper, we tackle the problem of hemivariational inequalities of parabolic type on locally finite graphs and we prove the existence of its weak solutions.
\end{abstract}

\maketitle \setlength{\textheight}{19.5 cm}
\setlength{\textwidth}{12.5 cm}
\newtheorem{theorem}{Theorem}[section]
\newtheorem{lemma}[theorem]{Lemma}
\newtheorem{proposition}[theorem]{Proposition}
\newtheorem{corollary}[theorem]{Corollary}
\theoremstyle{definition}
\newtheorem{definition}[theorem]{Definition}
\newtheorem{example}[theorem]{Example}
\theoremstyle{remark}
\newtheorem{remark}[theorem]{Remark}
\numberwithin{equation}{section} \setcounter{page}{1}

\newcommand{\LG}{\mathscr L_{\gamma,\kappa}^{\m G}}
\newcommand{\LGn}{\mathscr L_{\gamma,\kappa}^{\m G_n}}
\newcommand{\m}{\mathsf}
\newcommand{\la}{\langle}
\newcommand{\ra}{\rangle}
\newcommand{\ol}{\overline}
\newcommand{\ul}{\underline}



\section{Introduction}

Discrete calculus incorporates the various research works that focus on developing a proper theory for differential operators on discrete spaces with a net separation from the classical continuous calculus. From this perspective, discrete calculus should be differentiated from discretized calculus which concerns the discretization of the continuous framework for numerical and algorithmic purposes. Difference calculus, as a particular case of discrete calculus, is performed generally on the $d-$dimensional lattice graph (or grid) $\mathbb Z^d$ for some $d\geq 1$ and plays the role of an intermediate discipline. Discrete calculus aims then to establish a distinct and coherent core of calculus that operates purely in the discrete space without any reference to an underlying continuous counterpart. The philosophy behind this, is the fact that there is a solid connection between dynamics and the mathematical description of the space where they occur\cite{T76,GP10}.

The spaces of predilection for the discrete calculus are graphs and networks, from which cell complexes arise as general space structures \cite{BLW86}. The first application of graph theory to the modelling of physical systems came from Kirchhoff, who both developed the basic laws of circuit theory and also made fundamental contributions to graph theory\cite{K1847}. Among applications of modern graph theory one can mention manifold learning, filtering (denoising), content extraction,  ranking, clustering, and network characterization. The main technique here is to use the data to define weights on the network and then methods are used to formulate content extraction problems as convex energy minimization problems \cite{GA09,S99}. Nonconvex energy models appears also in data filtering on graphs with explicit discontinuities (rapid data change) \cite{MS89} and in nonsmooth nonconvex Regularizer in variational models for image restoration and segmentation \cite{GR92,JK14,N05}

After a decade from the development of the theory of nonlinear circuit networks in the sixties, operator theory on infinite graphs and the underlying Sobolev spaces began to be systematically developed as a theoretical core for studying elliptic and parabolic problems on graphs and networks. An embryonic study was initiated by M. Yamasaki and co-authors in \cite{NY76,Y75,Y77} and more elaborated work is  exposed by M.I. Ostrovskii in \cite{O05}. For some important use of the discrete version of  Sobolev spaces we refer to \cite{GP10} and references therein. The most modern expository on discrete operators and Sobolev spaces is the book of D. Mugnolo \cite{M14} where the central topic is the interplay of differential, difference operators and subdifferentials of convex functionals with the functional analytic theory of evolution equations together with combinatorial methods.

In many physical and social phenomena, the Laplacian operator arises naturally in the mathematical description of diffusion through discrete and continuous media. Discrete diffusion theory \cite{K64} based on discrete Fick's law \cite{HL09}, can be certainly considered as an approximation of its continuous counterpart, nonetheless problems still frequently arise where it would be advantageous  to have access to a diffusion theory valid specifically for discrete media \cite{K64}. The starting point for this theory is the formulation of Laplace operator on graphs and its associated discrete energy functional. Nakamura and Yamasaki introduced (for $\gamma\equiv1$ and $\kappa\equiv 0$) in \cite{NY76} on an infinite graph $\m G$ with node set $\m V$ the convex functional
\[
\mathscr E^p_{\gamma,\kappa}:\,\mathbb R^\m V\ni\phi\mapsto\frac{1}{p}\sum_{\substack{\m v,\m w\in\m V\\\m w\sim \m v}}\gamma(\m v,\m w)|\phi(\m v)-\phi(\m w)|^p+\frac{1}{p}\sum_{\m v\in\m  V}\kappa(\m v)|\phi(\m v)|^p\in[0,\infty]
\]The associated operator ($\LG:=\partial\mathscr E^p_{\gamma,\kappa}$) is nothing but the discrete $p-$Laplace operator. Let us note, parenthetically, the remark in \cite{M13}, that the development of the theory of nonlinear electric circuits and the theory of monotone operators and subdifferentials of convex functionals was simultaneous  by Minty and Rockafellar among others \cite{M60}. Different aspects of the discrete $p-$laplacian are studied in the literature \cite{GW16,KL10,M13} and found applications in nonlinear circuit theory, spectral clustering and image processing, sphere packing problem and with the tug-of-war-theory\cite{ELB08,EB17,GA09,HHK10,HL09,TBEL07,TA12} or emerging phenomena of a population of dynamically interacting units \cite{HHK10,TA12}. A systematic study of the Laplacian operators on graphs is achieved with means of discrete Dirichlet forms by D. Lenz and co-authors \cite{KL10,HKLW12} and references therein. For the Laplacian on finite weighted graphs with a nonlinear terms  we refer to \cite{HTM18} and references therein.

In the development of functional analysis on graphs, the finite difference operator plays a fundamental role. It started with an intuition that goes back to G. Boole \cite{B1860} revealing that the operator 
\[
\mathcal I^T \phi(\m v,\m w):=\phi(\m v)-\phi(\m w),\quad \phi\in\mathbb R^\m V
\]can be looked as a discretized version of the first derivative of the function $\phi$ defined in all points of the underlying graph $\m G$. In the definition of Sobolev spaces on graphs, the operator $\mathcal I^T \phi$ is an equivalent of the gradient. The parallel between discrete functional calculus and the classical continuous settings can be made by taking into account the following rules: Scalar functions replaces vectors of the node set, vector fields replace vectors of the edge set and gradient of scalar functions at some point replace evaluation of the difference operator at an edge.


The goal of this paper is to formulate a new class of variational-type inequalities consisting of nonmonotone multivalued perturbation of the discrete Laplacian on a locally finite graph. The pseudomonotone term is brought by a nonconvex functional defined  by an integral. Let $\m j:\mathbb R\rightarrow\mathbb R$ be a locally Lipschitz function whose Clarke's subdifferential satisfies a growth condition. The following sum functional
\begin{equation*}
\m J(\phi)=\sum_{\m v\in\m V}\mu(\m v)\m j(\phi(\m v))
\end{equation*}can be looked as a discretized version of the well-studied integral functionals of the form
\begin{equation*}
\m J(\phi)=\int_\mathscr O\m j(\phi(x))\,d\mu(x)
\end{equation*}
The primary question in such situation is the relation between the subdifferential of $\m j$ and $\m J$. In the integral functional case, this is what we commonly call Aubin-Clarke theorem. In this paper, we prove its discrete couterpart, that is
\begin{equation*}
\partial \m J(\phi)\subset\sum_{\m v\in\m V}\mu(\m v)\partial \m j(\phi(\m v))
\end{equation*}
Having this discrete version of Aubin-Clarke theorem at ones disposal one may formulate elliptic problem as follows

\begin{equation}
\la\LG\phi,\psi-\phi\ra+\sum_{\m v\in\m V}\mu(\m v)\m j^\circ(\phi(\m v);\psi(\m v)-\phi(\m v))\geq\la\m f,\psi-\phi\ra
\end{equation}
where $\m j^\circ$ is the generalized Clarke directional derivative of $\m j$ and its parabolic counterpart
\begin{equation}
\la\phi'+\LG\phi,\psi-\phi\ra+\sum_{\m v\in\m V}\mu(\m v)\m j^\circ(t,\phi(\m v);\psi(\m v)-\phi(\m v))\geq\la\m f,\psi-\phi\ra
\end{equation}Such problems will be called discrete hemivariational inequalities or hemivariational inequalities on graphs.

The general theory of hemivariational theory is a natural generalization of the classical variational theory where convex energy functionals are involved. Mathematical formulation of many engineering problems reveals cases that lack of monotonicity and corresponds to nonconvex superpotentials which cannot be formulated by the classical variational tools. By applying the mathematical notion of generalized gradient of Clarke \cite{C90}, Panagiotopoulos \cite{P93} introduced for the first time the so-called hemivariational inequalities. Since then, such formulation found applications in many fields, i.e. Navier-Stokes equations \cite{1MBA20,2MBA20}, boundary value problems\cite{AABE21,M04}, frictional contact \cite{MZ18}, history-dependent problems \cite{SMH18}, nonlocal problems \cite{ZT17} to name a few. Different methods are applied for the solvability of hemivariational inequalities, we can mention Galerkin approximation method, critical point theory, surjectivity theorems, extremal solutions method, Rothe approximation method , equilibrium problem method, penalty method..etc. The main assumption on the locally Lipschitz function include Rauch condition, growth condition or unilateral growth condition.

The remainder of the paper is structured as follows. In Section 2 we recall the functional setting on graphs and some concepts from nonsmooth analysis. In Section 3 we extend to the framework of locally finite graphs the Aubin-Clarke Theorem concerning the subdifferentiability of sum functionals. This may serve as a building block for developing variational methods for elliptic and parabolic problems involving the discrete Laplace operator and nonsmooth corresponding energy functional. In Section 3, we prove the existence and uniqueness of the elliptic hemivariational problem on locally finite graphs. The main tool is the well-know surjectivity result for pseudomonotone mappings. Section 4 is devoted to the discrete parabolic hemivariational inequality. The existence of a solution is reached by the use of a surjectivity result for the sum of maximal monotone and pseudomonotone mappings. In the last section, we provide some extensions to the problems discussed in previous sections. It concerns Galerkin scheme for discrete hemivariational inequalities, discrete variational-hemivariational inequalities and discrete quasi-hemivariational inequalities.

\section{Preliminaries}

\subsection{Sobolev spaces on graphs}For the concepts on graphs used in this paper and the underlying functional analysis which is the theoretical core we deploy in our formulations, we refer to the complete and self-contained book\cite{M14}.

Let $\m G=(\m V,\m E)$ be a direct graph, where $\m V$ is the set of nodes, which is finite or countable set and $E$ the set of edges, which is a subset of $\m V\times\m V$. A weighted graph is a quadruple $\m G=(\m E,\m V,\rho,\mu)$ where $(\m V,\m E)$ is a direct graph,  $\mu:\m V\rightarrow (0,\infty)$ is some given function and $\rho:\m E\rightarrow (0,\infty)$ is some other given function such that $\rho(\m e)=\rho(\bar{\m e})$ whenever $\m e,\,\bar{\m e}\in\m E$ ($\bar{\m e}=(\m w,\m v)$ when $\m e=(\m v,\m w)$). For $\m e=(\m v,\m w)$, we note $\m e_{-}:=\m v$ the initial endpoint of $\m e$ and $\m e_{+}:=\m w$ the terminal endpoint of $\m e$ and we say that they are adjacent (shortly $\m v\sim\m w$).

Define
\begin{equation*}
  \eta_{\m{ve}}^+=\left\{
  \begin{aligned}
1 & \quad\text{if }v\text{ is initial endpoint of }\m e\\
0 &\quad\text{otherwise}
\end{aligned}
 \right.
   ,\qquad  \eta_{\m{ve}}^-=\left\{ 
  \begin{aligned}
1 & \quad\text{if }v\text{ is terminal endpoint of }\m e\\
0 &\quad\text{otherwise}
\end{aligned}
 \right.
\end{equation*}

\begin{definition}
A weighted graph $\m G:=(\m V,\m E,\rho,\mu)$ is called outward locally finite if its outdegree function satisfies
\[
\m{deg}^+(\m v):=\sum_{\m e\in \m E}\eta_{\m{ve}}^+ \rho(\m e)\leq \m M_{\m v}^+\quad
\text{for all }\m v\in\m V\text{ and some }\m M_{\m v}^+>0.\]

It is called inward locally finite if its indegree function satisfies
\[
\m{deg}^-(\m v):=\sum_{\m e\in\m E}\eta_{\m{ve}}^- \rho(\m e)\leq\m M_{\m v}^-\quad
\text{for all }\m v\in V\text{ and some }\m M_{\m v}^->0.\]

It is locally finite if it is both inward and outward locally finite, i.e., if its degree function satisfies
\[
\m{deg}(\m v):= \m {deg}^+(\m v)+\m{deg}^-(\m v)\leq\m M_{\m v}\quad\text{ for all }\m v\in\m V \text{ and some }\m M_{\m v}.
\]
\end{definition}

\begin{example}
If $\m G$ is unweighted, then it is locally finite if and only if each node has only finitely incident edges.

\end{example}

Throughout this paper we suppose that $\m G:=(\m V,\m E,\rho,\mu)$ is a locally finite graph such that $\rho(\m e)>0$ for all $\m e\in\m E$, $\mu(\m v)>0$ for all $\m v\in\m V$ and $\mu(\m V)<\infty$. Let $p\in[1,+\infty)$, we denote by $\ell^p(\m E,\rho)$ the space of all functions $\varphi:\m E\rightarrow\mathbb R$ such that 
\[
\|\varphi\|_{\ell^p(\m E,\rho)}:=\left(\sum_{\m e\in\m E}|\varphi(\m e)|^p\rho(\m e)\right)^{1/p}<\infty
\]or else
\[
\|\varphi\|_{\ell^{\infty}(\m E,\rho)}:=\sup_{\m e\in\m E}|\varphi(\m e)|\rho(\m e)<\infty
\]For $p=2$,  $\ell^2(\m E,\rho)$ is a Hilbert space endowed with the inner product
\[
\la \varphi_1,\varphi_2\ra_\rho=\sum_{\m e\in\m E}\rho(\m e)\varphi_1(\m e)\varphi_2(\m e)
\] 
Similarly, we denote by $\ell^p(\m V,\mu)$ the space of all functions $\phi:\m V\rightarrow\mathbb R$ such that 
\[
\|\phi\|_{\ell^p(\m V,\mu)}:=\left(\sum_{\m v\in\m V}|\phi(\m v)|^p\mu(\m v)\right)^{1/p}<\infty
\]or else
\[
\|\phi\|_{\ell^\infty(\m V,\mu)}:=\sup_{\m v\in\m V}|\phi(\m v)|\mu(\m v)<\infty
\]
For $p=2$,  $\ell^2(\m V,\mu)$ is a Hilbert space endowed with the inner product
\[
\la \phi_1,\phi_2\ra_\mu=\sum_{\m v\in\m V}\mu(\m v)\phi_1(\m v)\phi_2(\m v)
\] 
We simply write $\ell^p(\m E)$ if $\rho\equiv 1$ and $\ell^p(\m V)$ if $\mu\equiv 1$.
Define 
\[
(\mathcal I^T \phi)(\m e)=\phi(\m e_{+})-\phi(\m e_{-}),\quad \phi\in\mathbb R^{\m V},\,\m e\in\m E
\] The difference operator $\mathcal I^T \phi$ can be looked as a discretized version of the first derivative of a function $\phi$ defined in all points of $G$. This plays a relevant role in the development of functional analysis on graphs. For $p\in [0,+\infty[$ we define the discrete Sobolev spaces of order one by 
\[
\m W^{1,p}_{\rho,\mu}(\m V)=\{\phi\in \ell^p(\m V,\mu):\,\mathcal I^T \phi\in \ell^p(\m E,\rho)\}
\]
The space $\m W^{1,p}_{\rho,\mu}(\m V)$ is a Banach space endowed with the norm
\[
\|\phi\|_{\m W^{1,p}_{\rho,\mu}}=\|\phi\|_{\ell^p(\m V,\mu)}+\|\mathcal I^T\phi\|_{\ell^p(\m E,\rho)}
\]and a Hilbert space for $p=2$ endowed with the inner product
\[
\la \phi,\psi\ra_{\rho,\mu}=\sum_{\m v\in \m V}\mu(\m v)\phi(\m v)\psi(\m v)+\sum_{\m e\in\m E}\rho(\m e)\left(\phi(\m e_{+})-\phi(\m e_{-})\right)\left(\psi(\m e_{+})-\psi(\m e_{-})\right)
\]

Recall that the distance $\m{dist}_\rho(\m v,\m w)$ of two nodes $\m v,\,\m w$ is defined as the infimum of the lengths of all paths from $\m v$ to $\m w$. In this way $\m G$, to be more precise, $\m V$ becomes  a metric space which is not complete in general unless $\rho$ is uniformly bounded away from $0$, i.e., $\frac{1}{\rho}\in \ell^{\infty}$. The ball of radius $r>0$ and center $\m v_0$ with respect to $\m{dist}_\rho$ is defined by
\[
\m B_\rho(\m v_0,r):=\{\m w\in\m V:\,\,\m{dist}_\rho(\m v_0,\m w)<r\}.
\]

If $\m G$ is connected, then by Proposition 38 (1) in \cite{M14}, the space $W^{1,p}_{\rho,\mu}(\m V)$ is densely and continuously embedded in $\ell^p(\m V,\mu)$ for all $1\leq p\leq\infty$. Let additionally $p<\infty$, then by Proposition 38 (2) in \cite{M14} this embedding is compact if for $\epsilon>0$ there are $\m v\in\m V$ and $r>0$ such that
\begin{enumerate}
\item[(i) ] $\m B_\rho(\m v,r)$ is a finite set
\item[(ii) ] there holds
\begin{equation}\label{rho}
\sum_{\m w\notin \m B_\rho(\m v,r)}|\phi(\m w)|^p\mu(\m w)<\epsilon^p
\end{equation}for all $\phi$ in the unit ball of $W^{1,p}_{\rho,\mu}(\m V)$.
\end{enumerate}

\begin{remark}
\begin{enumerate}
\item Condition (i) is satisfied if $\rho$ is uniformly bounded from below away from $0$, and in particular \eqref{rho} holds.
\item For all $\epsilon>0$, condition (ii) is satisfied for $r>\m{vol}_\rho(\m G):=\sum_{\m e\in\m E}\rho(\m e)$ if $\m{vol}_\rho(\m G)$ is finite.
\end{enumerate}
\end{remark}  
The space $\m W^{1,p}_{0,\rho,\mu}(\m V)$ is the closure of the space $\m C_{0}(\m V)$ of finitely supported functions on $\m V$ in the norm of $\m W^{1,p}_{\rho,\mu}(\m V)$ ($\m C_{0}(\m V)$ plays the role of test functions). For $1\leq p\leq \infty$, $\m W^{1,p}_{0,\rho,\mu}(\m V)$ is a Banach space with respect to the norm of $\m W^{1,p}_{\rho,\mu}(\m V)$, and a Hilbert space for $p=2$. For $1\leq p<\infty$ it is continuously and densely embedded into  $\ell^p(\m V,\mu)$. If $1\leq p<\infty$, then it is separable in $\ell^p(\m V,\mu)$ and if $1<p<\infty$ it is uniformly convex and hence reflexive \cite{M14}.

\subsection{Abstract surjectivity result}

Let $E$ be a reflexive Banach space with its dual $E^*$ and $A:D(A)\subset E\rightarrow 2^{E^*}$ be a multivalued function, where $D(A)=\{u\in E:\, Au\neq\emptyset\}$, stands for the domain of $A$. We say that $A$ is \textit{monotone} if $\la u^*-v^*,u-v\ra_{E^*\times E}\geq 0$ for all $u^*\in Au$,\, $v^*\in Av$  and  $u,\,v\in D(A)$. If moreover, $A$ has a maximal graph in the sense of inclusion among all monotone operators, then we say that $A$ is maximal monotone. We say that $A$ is \textit{pseudomonotone} operators if it satisfies the following properties,
\begin{enumerate}
\item[(a)] for each $u\in E$, the set $Au$ is nonempty, closed and convex in $E^*$.
\item[(b)] $A$ is upper semicontinuous from each finite dimensional subspace of $E$ into $E^*$ endowed with its weak topology;
\item[(c)] if $u_n\rightarrow u$ weakly weakly in $E$, $u_n^*\in Au_n$ and $\limsup_{n\to\infty}\limits\,\la u_n^*,u_n-u\ra_{E^*\times E}\leq 0$, then for each $v\in E$ there exists $v^*\in Au$ such that $\la v^*,u-v\ra_{E^*\times E}\leq \liminf_{n\to\infty}\limits\,\la u_n^*,u_n-v\ra_{E^*\times E} $.
\end{enumerate}
For a linear, maximal monotone operator $L:D(L)\subset E\rightarrow E^*$, an operator  $A$ is said to be \textit{pseudomonotone with respect to $D(L)$}(or $L-$pseudomonotone) if $(a)$ and $(b)$ are satisfied and
\begin{enumerate}
\item[(c')] for each sequences $\{u_n\}\subset D(L)$ and $\{u_n^*\}\subset E^*$ with $u_n\rightarrow u$  weakly in  $E$, $Lu_n\rightarrow Lu$  weakly in  $E^*$,  $u_n^*\in Au_n$  for all $n\in\mathbb N$, $u_n^*\rightarrow u^*$ weakly in $E^*$ and  $\displaystyle\limsup_{n\to+\infty}\la u_n^*,u_n-u\ra_{E^*\times E}\leq 0$, we have $u^*\in Au$ and $\displaystyle\lim_{n\to+\infty}\la u_n^*,u_n\ra_{E^*\times E}=\la u^*,u\ra_{E^*\times E}$.
\end{enumerate}

$A$ is \textit{coercive}  if there exists a function $c:\mathbb R^+\rightarrow\mathbb R$ with $c(r)\rightarrow\infty$ as $r\to\infty$ such that $\la u^*,u\ra_{E^*\times E}\geq c(\|u\|_E)\|u\|_E$ for every $(u,u^*)\in\mathrm{Graph}(A)$.

Now let $f:E\rightarrow \ol{\mathbb R}:=\mathbb R\cup\{+\infty\}$ be a proper, convex and lower semicontinuous functional. The mapping $\partial_c f:E\rightarrow 2^{E^*}$ defined by 
\[
\partial_cf(u)=\{u^*\in E^*:\,\la u^*,v-u\ra_{E^*\times E}\leq f(v)-f(u)\text{ for all }v\in E \},
\]is called the subdifferential of $f$. Any element $u^*\in \partial_cf(u)$ is called a subgradient of $f$ at $u$. It is a well know fact that $\partial f_c$ is a maximal monotone operator.

Let $F:E\rightarrow\mathbb R$ be a locally Lipschitz continuous functional and $u,\,v\in E$. We denote by $F^\circ(u;v)$ the generalized Clarke directional derivative of $F$ at the point $u$ in the direction $v$ defined by
\[
F^\circ(u;v)=\displaystyle\limsup_{w\to u,\,t\downarrow 0}\frac{F(w+tv)-F(w)}{t}.
\]
The generalized Clarke gradient $\partial F:E\rightarrow 2^{E^*}$ of $F$ at $u\in E$ is defined by
\[
\partial F(u)=\{\xi\in E^*:\,\la\xi,v\ra_{E^*\times E}\leq F^\circ(u;v)\text{ for all } v\in E\}.
\]
We collect the following properties
\begin{enumerate}
\item[(a)] the function $v\mapsto F^\circ(u;v)$ is positively homogeneous, subadditive and satisfies
\[
|F^\circ(u;v)|\leq L_u\|v\|_E\text {for all } v\in E,
\]where $L_u>0$ is the rank of $F$ near $u$.
\item[(b)] $(u,v)\mapsto F^\circ(u;v)$ is upper semicontinuous.
\item[(c)] $\partial F(u)$ is a nonempty, convex and weakly$^*$ compact subset of $E^*$ with $\|\xi\|_{E^*}\leq L_u$ for all $\xi\in\partial F(u)$.
\item[(d)] for all $v\in E$, we have $F^\circ(u;v)=\max\{\la\xi,v\ra_{E^*\times E}:\,\xi\in\partial F(u)\}$.

\end{enumerate}

We say that a function $F:E\rightarrow \mathbb R$ is regular at $x$, if for all $v$, the usual one-sided directional derivative \[F'(x,v):=\displaystyle\lim_{h\downarrow0}\frac{F(x+hv)-F(x)}{h}\] exists and is equal to the generalized directional derivative $F^\circ(x;v)$. By Proposition 2.3.6 in \cite{C90}, if $F$ is locally Lipschitz and convex, then it is regular at any $x$.

The following surjectivity result for operators which are $L-$pseudomonotone will be used in our existence theorems in Section 3 and Section 4 (cf. \cite[Theorem 2.1]{PPR99}).
\begin{theorem}\label{thm0}
If $E$ is a reflexive strictly convex Banach space, $L:D(L)\subset E\rightarrow E^*$ is a linear maximal monotone operator, and $A:E\rightarrow 2^{E^*}$ is a multivalued operator , which is bounded, coercive and $L-$pseudomonotone. Then $L+A$ is a surjective operator, i.e. for all $f\in E^*$, there exists $u\in E$ such that $Lu+Au\ni f$. 
\end{theorem}

It is worth to mention that one can drop the strict convexity of the reflexive Banach space $E$. It suffices to invoke the Troyanski renorming theorem to get an equivalent norm so that the space itself and its dual are strictly convex(cf. \cite[Proposition 32.23, p.862]{Z90}).

\section{Nonconvex sum functionals on graphs}

In this section we will prove the discrete counterpart of the Aubin-Clarke theorem concerning the subdifferentiability of nonconvex sum functionals. We consider a function $\m j:\mathbb R\rightarrow\mathbb R$ which satisfies the following hypothesis $\m H(\m j)$:

\begin{enumerate}
\item[$\m H(\m j)_1$] $\m j:\mathbb R\rightarrow\mathbb R$ is locally Lipschitz.
\item[$\m H(\m j)_2$] there exists $\alpha_\m j>0$ such that
\[
|z|\leq \alpha_\m j(1+|s|^{}),\qquad \forall z\in\partial \m j(s)
\]
\end{enumerate}

Next we define the superpotential $\m J:\ell^2(\m V,\mu)\rightarrow\mathbb R$ defined by
\[
\m J(\phi)=\sum_{\m v\in\m V}\mu(\m v)\, \m j(\phi(\m v))
\]for all $\phi\in \ell^2(\m V,\mu)$.  The sum functional $\m J$ can be seen as the discrete version of the classical integral functionals. The following results is the discret version of the Aubin-Clarke theorem \cite{C90}.

\begin{proposition}\label{prop1}
Under the assumption $\m H(\m j)$:
\begin{enumerate}
\item The functional $\m J$ is well defined and finite on $\ell^2(\m V,\mu)$.
\item $\m J$ is locally Lipschitz.
\item For all $\phi,\,\psi\in \ell^2(\m V,\mu)$, we have
\begin{equation}\label{eqj0}
\m J^0(\phi;\psi)\leq \sum_{\m v\in\m V}\mu(\m v)\, \m j^0(\phi(\m v);\psi(\m v))
\end{equation}
\item For all $\phi\in \ell^2(\m V,\mu)$ we have
\[
\partial \m J(\phi)\subset \sum_{\m v\in\m V}\mu(\m v)\partial \m j(\phi(\m v))
\]
This inclusion is understood in the sense that for each $\phi^*\in\partial \m J(\phi)\subset \ell^2(\m V,\mu)$, there exists a mapping $\m V\ni\m v\mapsto\xi(\m v)$ such that $\xi(\m v)\in\partial \m j(\phi(\m v))$ and 
\[
\la \phi^*,\psi\ra=\sum_{\m v\in\m V}\mu(\m v)\xi(\m v) \psi(\m v) 
\]for all $\psi\in \ell^2(\m V,\mu)$.
\end{enumerate}
\end{proposition}
\begin{proof}
By $\m H(\m j)$, Lebourg's mean value theorem and Hölder's inequality, we have 
\begin{align*}
|\m J(\phi_1)-\m J(\phi_2)|&\leq \sum_{\m v\in \m V}\mu(\m v)\,|\,\m j(\phi_1(\m v))-\m j(\phi_2(\m v))\,|\\
&\leq \sum_{\m v\in \m V}\mu(\m v)|\xi|\,|\phi_1(\m v)-\phi_2(\m v)|\quad(\xi\in\partial\m j(s)\text{ with } s\in[\phi_1(\m v),\phi_2(\m v)])\\
&\leq \alpha_\m j \sum_{\m v\in \m V}\mu(\m v)\left(1+|\phi_1(\m v)|+|\phi_2(\m v)|\right)\,|\phi_1(\m v)-\phi_2(\m v)|\\
&\leq \alpha_\m j \left(\sum_{\m v\in \m V}\mu(\m v)(1+|\phi_1(\m v)|+|\phi_2(\m v)|)^2\right)^{1/2}\|\phi_1-\phi_1\|_{\ell^2(\m V,\mu)}\\
&\leq \alpha_\m J\left(1+\|\phi_1\|_{\ell^2(\m V,\mu)}+\|\phi_2\|_{\ell^2(\m V,\mu)}\right)\,\|\phi_1-\phi_2\|_{\ell^2(\m V,\mu)}\\
&\leq \alpha_\m j' \,\|\phi_1-\phi_2\|_{\ell^2(\m V,\mu)}
\end{align*}where $\alpha_\m J$ depends only on $\alpha_\m j$, $\mu$ and $m$ where $m$ is such that $\|\phi_1\|_{\ell(\m V,\mu)},\,\|\phi_2\|_{\ell(\m V,\mu)}\leq m$. Consequently, the functional $\m J$ is well-defined, finite and locally Lipschitz.

Let $\phi,\,\psi\in \ell^2(\m V,\mu)$, by Fatou's Lemma with counting measure,
\begin{align*}
\m J^0(\phi;\psi)&=\limsup_{\theta\to\phi,\,\lambda\to 0}\frac{\m J(\theta+\lambda\psi)-\m J(\theta)}{\lambda}\\
 &=\limsup_{\theta\to\phi,\,\lambda\to 0}\frac{1}{\lambda}\sum_{\m v\in\m V}\mu(\m v)\left(\m j(\theta(\m v)+\lambda\psi(\m v))-\m j(\theta(\m v))\right)\\
&\leq \sum_{\m v\in\m V}\mu(\m v)\limsup_{\theta\to\phi,\,\lambda\to 0}\frac{\m j(\theta(\m v)+\lambda\psi(\m v))-\m j(\theta(\m v))}{\lambda}\\
&\leq \sum_{\m v\in\m V}\mu(\m v)\m j^0(\phi(\m v);\psi(\m v))
\end{align*}
Now, define $\widehat{\m j}$ and $\widehat{\m J}$ as follows
\begin{equation*}
\widehat{\m j}(\psi)=\m j^0(\phi;\psi),\qquad
\widehat{\m J}(\psi)=\sum_{\m v\in\m V}\mu(\m v)\widehat{\m j}(\psi(\m v)).
\end{equation*}It is clear that $\widehat{\m j}$ is convex and thus, so is $\widehat{\m J}$. If we observe that $\widehat{\m j}(0)=\widehat{\m j}(0)=0$, we have $\widehat{\m J}(\psi)-\widehat{\m J}(0)\geq \la\xi,\psi\ra_{\ell^2(\m V,\mu)}$ for all $\psi$ and $\xi\in\partial\widehat{\m J}(0)$. Since $\partial\widehat{\m J}(0)\subset\sum_{\m v\in\m V}\mu(\m v)\partial\widehat{\m j}(0)$, for convex functions, see \cite{IL72}. Then there exists a map $\m v\mapsto\xi(\m v)$ with $\xi(\m v)\in \partial\widehat{\m j}(0)$ such that for every $\theta\in \ell^2(\m V,\mu)$
\[
\la\xi,\theta\ra_{\ell^2(\m V,\mu)}=\sum_{\m v\in \m V}\mu(\m v)\xi(\m v)\theta(\m v)
\]However, $\partial\widehat{\m j}(0)=\partial\m j(\phi)$, so the result would follow.

\end{proof}

Remark that if either $\m j$ or $-\m j$ is regular, then $\m J$ is also regular and equality \eqref{eqj0} in holds true. In fact, one have
\begin{align*}
\m J^\circ(\phi;\psi) &=\displaystyle\limsup_{\theta\to \phi,h\downarrow 0}\frac{\m J(\theta+h\psi)-\m J(\psi)}{h}\\
\geq &\displaystyle\lim_{h\downarrow 0} \frac{\m J(\phi+h\psi)-\m J(\phi)}{h} \\
=&\displaystyle\lim_{h\downarrow 0} \sum_{\m v\in \m V}\mu(\m v)\frac{\m j(\phi(\m v)+h\psi(\m v))-\m j(\phi(\m v))}{h}\\
=&  \sum_{\m v\in \m V}\mu(\m v)\displaystyle\lim_{h\downarrow 0}\frac{\m j(\phi(\m v)+h\psi(\m v))-\m j(\phi(\m v))}{h}\\
=&  \sum_{\m v\in \m V}\mu(\m v)\m j'(\phi(\m v);\psi(\m v))\\
=& \sum_{\m v\in \m V}\mu(\m v)\m j^\circ(\phi(\m v);\psi(\m v))
\end{align*}which, by Proposition \ref{prop1}, leads to 
\[
\m J^\circ(\phi;\psi)= \sum_{\m v\in \m V}\mu(\m v)\m j^\circ(\phi(\m v);\psi(\m v))
\]
Moreover, $\m J^\circ=\m J'$ since
\begin{align*}
\m J'(\phi;\psi)=&\displaystyle\lim_{h\downarrow 0}\sum_{\m v\in\m V}\mu(\m v)\frac{\m j(\phi(\m v)+h\psi(\m v))-\m j(\phi(\m v))}{h}\\
=& \sum_{\m v\in\m V}\mu(\m v) \m j'(\phi(\m v);\psi(\m v))\\
=& \m J^\circ(\phi;\psi)
\end{align*}

\begin{proposition}\label{prop2}
Under hypothesis $\m H(\m j)$, the following inequalities hold
\[
\m J^0(\phi;\psi)\leq \alpha_\m J\left(1+\|\phi\|^{}_{\ell^2(\m V,\mu)}\right)\|\psi\|_{\ell^2(\m V)},\quad\forall \phi,\,\psi\in \ell^2(\m V,\mu)
\]and 
\[
\|\theta\|_{\ell^2(\m V,\mu)}\leq \alpha_\m J\left(1+\|\phi\|^{}_{\ell^2(\m V,\mu)}\right),\quad \forall \theta\in\partial(\m J_{|\ell^2(\m V,\mu)})(\phi),\,\phi\in \ell^2(\m V,\mu)
\]
\end{proposition}
\begin{proof}Let $\phi,\,\psi\in \ell^2(\m V,\mu)$, we have
\begin{align*}
\m J^0(\phi;\psi)  &  \leq \sum_{\m v\in \m V}\mu(\m v)\m j^0(\phi(\m v);\psi(\m v))\\
                    &  = \sum_{\m v\in \m V}\mu(\m v)\max\{\theta.\psi(\m v)\,|\, \theta\in\partial\m j(\phi(\m v))\}\\
                    & = \sum_{\m v\in \m V}\mu(\m v)\max\{|\theta|\,|\psi(\m v)|\,|\, \theta\in\partial\m j(\phi(\m v))\}
\end{align*}
By $\m H(\m j)$, we have
\begin{align*}
 \m J^0(\phi;\psi)   & \leq \alpha_\m j\sum_{\m v\in \m V}\mu(\m v)\left(1+|\phi(\m v)|\right)|\psi(\m v)|\\
                    & \leq \alpha_\m j\sum_{\m v\in \m V}\sqrt{\mu(\m v)}\left(1+|\phi(\m v)|\right)\, \sqrt{\mu(\m v)}|\psi(\m v)|\\
                    & \leq \alpha_\m j\left(\sum_{\m v\in \m V}\mu(\m v)\left(1+|\phi(\m v)|\right)^2\right)^{1/2}\, \|\psi(\m v)\|_{\ell^2(\m V,\mu)}\\
                    & \leq \alpha_\m j'\left(1+\|\phi\|_{\ell^2(\m V,\mu)}\right)\, \|\psi(\m v)\|_{\ell^2(\m V,\mu)}
\end{align*}
This, together with \cite[Proposition 2.1.2]{C90}, yield to
\begin{align*}
\|\theta\|_{\ell^2(\m V,\mu)}&=\sup\{\la\theta,\psi\ra_{\ell^2(\m V,\mu)}\,|\,\|\psi\|_{\ell^2(\m V,\mu)}\leq 1\}\\
&\leq \sup\{\m J^0(\phi;\psi)\,|\,\|\psi\|_{\ell^2(\m V,\mu)}\leq 1\}\\
&\leq \alpha_\m j' \left(1+\|\phi\|_{\ell^2(\m V,\mu)}\right),\quad\text{for }\theta\in\partial\m J_{|\ell^2(\m V,\mu)}(\phi),\,\phi\in \ell^2(\m V,\mu)
\end{align*}

\end{proof}
In what follows we consider a superpotential $\m j$ which  subdifferential is obtained by "filling in the gaps" procedure \cite{R77}. For $(\m v,t)\in \m V\times\mathbb R$, define
\[
\m j(\m v,t)=\int_0^t\beta(\m v,s)\,ds
\]where $\beta:\m V\times\mathbb R\rightarrow\mathbb R$ is a function such that $\beta(\m v,.)$ is measurable for all $\m v\in \m V$ and satisfies the following growth condition
\[
|\beta(\m v,t)|\leq \alpha_\beta(1+|t|)
\]for all $\m v\in \m V$ and a.e $t\in\mathbb R$. Note that $\m j(\m v,.)$ is locally Lipschitz and satisfies a growth condition with eventually different constant.

We present the discrete version of the functional in Section 2 of \cite{C81} which is in the following form
\[
\m J(\phi)=\sum_{\m v\in\m V}\mu(\m v)\int_0^{\phi(\m v)}\beta(\m v,t)\,dt
\]Then $\m J$ is a locally Lipschitz function defined on $\ell^2(\m V,\mu)$.

Let's first describe the "filling in the gaps" procedure. Let $\theta \in L^\infty_{loc}(\mathbb R)$, for $\epsilon>0$ and $t\in \mathbb R$, we define:
\begin{equation*}
\underline{\theta}_\epsilon(t)=\displaystyle\essinf_{|t-s|\leq\epsilon}\theta(s),\quad  \quad \overline{\theta}_\epsilon(t)=\displaystyle\esssup_{|t-s|\leq\epsilon}\theta(s).
\end{equation*}
For a fixed $t\in \mathbb R$, the functions $\underline{\theta}_\epsilon,  \overline{\theta}_\epsilon$ are decreasing and increasing in $\epsilon$, respectively. Let 
\begin{align*}
\underline{\theta}(t)=\lim_{\epsilon \to 0^+} \underline{\theta}_\epsilon(t), \quad\quad \overline{\theta}(t)=\lim_{\epsilon \to 0^+}\overline{\theta}_\epsilon(t),
\end{align*}
and let $\hat{\theta}(t):\mathbb R \to 2^\mathbb R$ be a multifunction defined by
\begin{equation*}
\widehat{\theta}(t)=\left[ \underline{\theta}(t),\overline{\theta}(t) \right]
\end{equation*}
From Chang \cite{C81} we know that a locally Lipschitz
function $j:\mathbb R \to \mathbb R$ can be determined up to an additive constant by the relation $$j(t)=\int_0^t \theta(s)\, d s $$ such that $\partial j(t) \subset \hat{\theta}(t)$ for all $t\in\mathbb R$. If moreover, the limits $\theta(t\pm 0)$ exist for every $t\in\mathbb R$, then $\partial j(t) = \hat{\theta}(t)$.

Now by Proposition \ref{prop1}, we have
\[
\partial\m J(\phi)\subset \sum_{\m v\in\m V}\mu(\m v)\hat\beta(\m v,\phi(\m v))
\]If we note $\underline{\m J}(\phi)= \sum_{\m v\in\m V}\mu(\m v)\underline{\beta}(\m v,\phi(\m v))$ and $\overline{\m J}(\phi)= \sum_{\m v\in\m V}\mu(\m v)\overline{\beta}(\m v,\phi(\m v))$ we can write 
\begin{equation}\label{eqJ}
\partial\m J(\phi)\subset [\underline{\m J}(\phi),\overline{\m J}(\phi)],\quad\text{for all }\phi\in\ell^2(\m V,\mu)
\end{equation}If $\m J$ is convex, then it is regular and equality in \eqref{eqJ} holds true.

\section{Existence Result for elliptic problem}

Let $\m G=(\m V,\m E,\rho,\mu)$ be a weighted direct graph. Let $\m j:\mathbb R\rightarrow\mathbb R$ be a locally Lipschitz function and $\m j^0(.;.)$ denotes its generalized directional derivative. Consider a function $\gamma:\m E\rightarrow(0,\infty)$ such that $\gamma(\m e)=\gamma(\bar{\m e})$ and a function $\kappa:\m V\rightarrow (0,\infty)$ such that the following assumption $\m H(\m G)$ hold:
\begin{enumerate}
\item[$(\m G_1)$] There exist $\underline\alpha_\gamma,\,\overline\alpha_\gamma>0$ such that \[\underline\alpha_\gamma\rho(\m e)\leq \gamma(\m e)\leq \overline\alpha_\gamma\rho(\m e),\quad\text{for all }\m e\in\m E.\]
\item[$(\m G_2)$] There exist $\underline\alpha_\mu,\,\overline \alpha_\mu>0$ such that \[\underline\alpha_\mu\kappa(\m v)\leq \mu(\m v)\leq \overline\alpha\mu\kappa(\m v),\quad\text{for all }\m v\in\m V.\]
\end{enumerate}

We denote
$\m W_0:=\m W^{1,2}_{0,\rho,\mu}(\m V)$ and we define the operator $\mathscr L_{\gamma,\kappa}^{\m G}:\m W_0\rightarrow\m W_0$ by 
\[
(\LG \phi)(\m v):=\frac{1}{\mu(\m v)}\sum_{\substack{\m w\in\m V\\\m w\sim \m v}}\gamma(\m v,\m w)(\phi(\m v)-\phi(\m w))+\frac{\kappa(\m v)}{\mu(\m v)}\phi(\m v),\quad\m v\in\m V,\, \phi\in \m W_0
\]

The aim of this paper is to prove the existence of solutions to the  problem of finding $\phi$ such that

\begin{equation}\label{prob0}
\LG \phi+\partial \m J (\phi)\ni \m f,\quad
\phi\in \m W_0
\end{equation}
which is equivalent to finding $\phi\in \m W_0$ such that
\begin{equation}\label{prob1}
  \left\{
  \begin{aligned}
&\LG \phi+\xi = \m f,\\[0.15cm]
&\xi\in \partial \m J (\phi).
\end{aligned}
 \right.
\end{equation}

To obtain a variational formulation of Problem \eqref{prob1}, we multiply by $\psi-\phi$ and we use the definition of $\partial J$. This produces the following hemivariational inequality: Find $\phi\in \m W_0$ such that for every $\psi\in \m W_0$
\begin{equation}\label{prob3}
\langle \LG \phi-\m f,\psi-\phi\rangle +\sum_{\m v\in\m V}\mu\,(\m v)\m j^0(\phi(\m v);\psi(\m v)-\phi(\m v))\geq 0
\end{equation}
Note that problems \eqref{prob0} and \eqref{prob3} are not equivalent. This will be the case if, for example, either $\m j$ or $-\m j$ is regular. Generally, if $\phi$ is a solution of Problem \eqref{prob00}, then it is a solution of Problem \eqref{prob3}.

\begin{definition}
We say that $\phi\in \m W_0$ is a weak solution to problem \eqref{prob0}, if
\[
 \sum_{\m v\in \m V} \mu(\m v)(\LG \phi)(\m v)(\psi(\m v)-\phi(\m v)) +\sum_{\m v\in\m V}\mu(\m v)\m j^0(\phi(\m v);\psi(\m v)-\phi(\m v))\geq \sum_{\m v\in \m V}\mu(\m v)\m f(\m v)(\psi(\m v)-\phi(\m v))
\]for all $\psi\in \m W_0$.

\end{definition}

\begin{theorem}\label{smallness}
Under assumptions $\m H(\m G)$ and $\m H(\m j)$ with
\begin{equation}
\alpha_\m J<\frac{1}{2}\underline\alpha_\gamma\wedge\underline\alpha_\mu,
\end{equation}the problem \eqref{prob3} has at least one weak solution.
\end{theorem}
\begin{proof}We observe first that 
\begin{align*}
\la \LG\phi,\psi \ra &= \sum_{\m v\in \m V}\mu(\m v)(\LG\phi)(\m v)\psi(\m v)\\
&= \sum_{\m v\in \m V} \sum_{\substack{\m w\in\m V\\\m w\sim \m v}}\gamma(\m v,\m w)(\phi(\m v)-\phi(\m w))\psi(\m v)+\sum_{\m v\in \m V}\kappa(\m v)\phi(\m v)\psi(\m v) \\
&= \frac{1}{2}\sum_{(\m v,\m w)\in \m E} \gamma(\m v,\m w)(\phi(\m v)-\phi(\m w))(\psi(\m v)-\psi(\m w))+\sum_{\m v\in \m V}\kappa(\m v)\phi(\m v)\psi(\m v)
\end{align*}
From one side, we have
\begin{align}\label{ineqLG1}\nonumber
\la \LG\phi,\phi\ra &=  \frac{1}{2}\sum_{(\m v,\m w)\in \m E} \gamma(\m v,\m w)(\phi(\m v)-\phi(\m w))^2+\sum_{\m v\in \m V}\kappa(\m v)\phi^2(\m v) \\\nonumber
&=\frac{1}{2}\sum_{(\m v,\m w)\in \m E} \frac{\gamma(\m v,\m w)}{\rho(\m v,\m w)}\rho(\m v,\m w)(\phi(\m v)-\phi(\m w))^2+\sum_{\m v\in \m V}\frac{\kappa(\m v)}{\mu(\m  v)}\mu(\m v)\phi^2(\m v)\\
&\geq \frac{1}{2}\underline\alpha_\gamma\wedge\underline\alpha_\mu\,\|\phi\|_{\m W_0}^2
\end{align}Thus, $\LG$ is strongly monotone and coercive. From another side
\begin{align}\label{ineqLG2}\nonumber
\la \LG\phi,\psi\ra & =  \frac{1}{2}\sum_{(\m v,\m w)\in \m E} \frac{\gamma(\m v,\m w)}{\rho(\m v,\m w)}\sqrt{\rho(\m v,\m w)}(\phi(\m v)-\phi(\m w))\sqrt{\rho(\m v,\m w)}(\psi(\m v)-\psi(\m w))\\\nonumber
&\qquad\qquad\qquad\qquad+\sum_{\m v\in \m V}\frac{\kappa(\m v)}{\mu(\m v)}\sqrt{\mu(\m v)}\phi(\m v)\sqrt{\mu(\m v)}\psi(\m v)   \\\nonumber
& \leq \frac{1}{2}\overline\alpha_\gamma \|\mathcal I^T\phi\|_{\ell^2(\m E,\rho)} \|\mathcal I^T\psi\|_{\ell^2(\m E,\rho)} +  \overline\alpha_\mu \|\phi\|_{\ell^2(\m V,\mu)}\|\psi\|_{\ell^2(\m V,\mu)}\\
&\leq \frac{1}{2}\overline\alpha_\gamma\vee\overline\alpha_\mu\,\|\phi\|_{\m W_0}\|\psi\|_{\m W_0}
\end{align}
Thus, $\LG$ is continuous.

\textbf{Claim 1:} The operator $\LG+\partial\m J$ is coercive.

We have
\begin{align*}
\inf\{\la \LG\phi+\xi,\phi\ra\,|\,\xi\in\partial\m J(\phi)\}&= \la \LG\phi,\phi\ra+\inf\{\la\xi,\phi\ra_{\ell^2(\m V,\mu )}\,|\,\xi\in\partial\m J(\phi)\}\\
&\geq \frac{1}{2}\underline\alpha_\gamma\wedge\underline\alpha_\mu\|\phi\|_{\m W_0}^2-\sup\{\|\xi\|_{\ell^{2}(\m V,\mu)}\,|\,\xi\in\partial \m J(\phi)\}\|\phi\|_{\ell^2(\m V,\mu)}\\
&\geq \frac{1}{2}\underline\alpha_\gamma\wedge\underline\alpha_\mu\|\phi\|_{\m W_0}^2-\alpha_\m J\|\phi\|_{\ell^2(\m V,\mu)}-\alpha_\m J\|\phi\|^2_{\ell^p(\m V,\mu)}\\
&\geq \frac{1}{2}\underline\alpha_\gamma\wedge\underline\alpha_\mu\|\phi\|_{\m W_0}^2-\alpha_\m J\|\phi\|_{\m W_0}-\alpha_\m J\|\phi\|^2_{\m W_0}\\
&\geq \left(\frac{1}{2}\underline\alpha_\gamma\wedge\underline\alpha_\mu-\alpha_\m J\right)\|\phi\|_{\m W_0}^2-\alpha_\m J \|\phi\|_{\m W_0}.
\end{align*}
If $\frac{1}{2}\underline\alpha_\gamma\wedge\underline\alpha_\mu>\alpha_\m J$, the above inequality implies that $\LG+\partial \m J: \m W_0\rightarrow \m W_0^*$ is coercive.

\textbf{Claim 2:} The operator $\LG+\partial\m J$ is pseudomonotone. 

We know that $\partial\m J$ is nonempty, convex, weak-compact subset of $\m W_0$. Then for each $\phi\in \m W_0$, $\LG(\phi)+\partial\m J(\phi)$ is nonempty, bounded, closed and convex subset of $\m W_0$. Moreover, $\LG(\phi)+\partial\m J(\phi)$ is upper semicontinuous from $\m W_0$ to $w-\m W_0$.

Let $\phi_k$ be a sequence in $\m W_0$ converging weakly to $\phi$, and $\xi_k\in\partial\m J(\phi_k)$ such that
\[
\limsup_{k\to\infty}\la \LG(\phi_k)+\xi_k,\phi_k-\phi \ra\leq 0
\]which implies
\begin{equation}\label{ineq1}
\limsup_{k\to\infty}\la \LG(\phi_k),\phi_k-\phi \ra+\liminf_{k\to\infty}\la\xi_k,\phi_k-\phi \ra\leq 0
\end{equation}
The embedding $\m W_0\hookrightarrow \ell^2(\m V,\mu)$ is compact. Therefore, $\phi_k$ converges strongly in $\ell^2(\m V,\mu)$ to $\phi$. By applying Theorem 2.2 in \cite{C81}, we have
\[
\partial\left(\m J_{|\m W_0}\right)(\phi)\subset \partial\left(\m J_{|\ell^2(\m V,\mu)}\right)(\phi),\quad\forall \phi\in \m W_0
\]Therefore,
\[
|\la \theta_k,\phi_k-\phi\ra_{\m W_0}|\leq \mathrm{const}\|\theta_k\|_{\ell^2(\m V,\mu)}\|\phi_k-\phi\|_{\ell^2(\m V,\mu)}.
\]
Thus
\[
|\la \theta_k,\phi_k-\phi\ra_{\m W_0}|\rightarrow 0,\quad\text{as }k\to\infty
\]
Then, from \eqref{ineq1} we have
\[
\limsup_{k\to\infty}\la \LG(\phi_k),\phi_k-\phi \ra\leq 0
\]We have from $\phi_k\to\phi$ weakly in $\m W_0$
\[
\limsup_{k\to\infty}\la \LG(\phi_k)-\LG(\phi),\phi_k-\phi \ra\leq 0
\]By the coercivity of $\LG$, we get
\[
\limsup_{k\to\infty}\|\phi_k-\phi\|_{\m W_0}\leq 0
\]Therefore we obtain
\begin{align*}
\phi_k&\rightarrow\phi\quad\text{strongly in }\m W_0\\
\LG\phi_k&\rightarrow\LG\phi\quad\text{strongly in }\m W_0.
\end{align*}It follows that there exist $\xi$ such that $\xi_k\to\xi$ weakly${}^*$ in $W_0$ and
\[
\lim_{k\to\infty}\la \LG(\phi_k)+\xi_k,\phi_k-\psi \ra= \la \LG(\phi)+\xi,\phi-\psi \ra
\]This implies that $\LG+\partial\m J: \m W_0\rightarrow \m W_0$ is pseudomonotone, which completes the proof.
\end{proof}
Let's consider the following additional assumption
\begin{enumerate}
\item[$\m H(\m j^0)$ ] There exists $\alpha_{\m j^0}>0$ such that
\[
\m j^0(s;t-s)+\m j^0(t;s-t)\leq \alpha_{\m j^0}|t-s|^2
\]for all $s,\,t\in\mathbb R$. 
\end{enumerate}
\begin{theorem}
Under assumptions $\m H(\m G)$, $\m H(\m j)$ and $\m H(\m j^0)$ with \[\alpha_{\m j^0}\vee\alpha_\m j<\frac{1}{2}\underline\alpha_\gamma\wedge\underline\alpha_\mu,\] the weak solution of the Problem \eqref{prob0} is unique. 
\end{theorem}
\begin{proof}
Let $\phi_1$ and $\phi_2$ be two weak solutions of Problem \eqref{prob0}. It follows then that 
\begin{align}\label{uni1}
\langle \LG \phi_1,\psi-\phi_1\rangle +\sum_{\m v\in\m V}\mu\,(\m v)\m j^0(\phi_1(\m v);\psi(\m v)-\phi_1(\m v))\geq \la \m f,\psi-\phi_1\rangle,\,\text{for all }\phi\\\label{uni2}
\langle \LG \phi_2,\psi-\phi_2\rangle +\sum_{\m v\in\m V}\mu\,(\m v)\m j^0(\phi_2(\m v);\psi(\m v)-\phi_2(\m v))\geq \la \m f,\psi-\phi_2\rangle, \,\text{for all }\phi
\end{align}
We replace $\phi$ in \eqref{uni1} by $\phi_2$ and in \eqref{uni2} by $\phi_1$ and we sum up the two inequalities. We obtain that 
\[
\langle \LG (\phi_1-\phi_2),\phi_2-\phi_1\rangle +\sum_{\m v\in\m V}\mu\,(\m v)\{\m j^0(\phi_1(\m v);\phi_2(\m v)-\phi_1(\m v))+j^0(\phi_2(\m v);\phi_1(\m v)-\phi_2(\m v))\}\geq 0
\]
It follows that 
\[
-\frac{1}{2}\underline\alpha_\gamma\wedge\underline\alpha_\mu\|\phi_2-\phi_1\|^2_{\m W_0}+\alpha_{\m j^0}\sum_{\m v\in\m V}\mu\,(\m v)|\phi_2(\m v)-\phi_1(\m v)|^2\geq 0
\]Thus
\[
(\alpha_{\m j^0}-\frac{1}{2}\underline\alpha_\gamma\wedge\underline\alpha_\mu)\|\phi_2-\phi_1\|^2_{\m W_0}\geq 0
\]If $\alpha_{\m j^0}<\frac{1}{2}\underline\alpha_\gamma\wedge\underline\alpha_\mu$, we have $\phi_1=\phi_2$, which completes the proof.

\end{proof}

\section{Existence Result for parabolic problem}
For $0<\m T<\infty$, we denote by $\mathcal V:=L^2(0,\m T;\ell^2(\m V,\mu))$ the usual time Sobolev space endowed with the norm
\[
\|\phi\|_{\mathcal V}=\left(\int_0^{\m T}\|\phi(t)\|^2_{\ell^2(\m V,\mu)}\,dt\right)^{1/2}
\]and consider a function $\m j:(0,\m T)\times\mathbb R\rightarrow\mathbb R$ which satisfies the following hypothesis $\m H(\m j)$:

\begin{enumerate}
\item[$\m H(\m j)_1$] for each $r\in\mathbb R$, the function $t\mapsto \m j(t,r)$ is measurable on $(0,\m T)$,
\item[$\m H(\m j)_2$] for a.e. $t\in(0,\m T)$, the functional $r\mapsto \m j(t,r)$ is locally Lipschitz,
\item[$\m H(\m j)_3$] there exists $\alpha_\m j>0$ such that
\[
|z|\leq \alpha_\m j(1+|r|^{}),\qquad \text{for all } z\in\partial \m j(t,r)\text{ and a.e. }t\in(0,\m T).
\]
\end{enumerate}

We define the superpotential $\m J:\mathcal V\rightarrow\mathbb R$ defined by
\[
\m J(\phi)=\sum_{\m v\in\m V}\mu(\m v)\, \int_0^{\m T}\m j(t,\phi(t,\m v))\,dt
\]for all $\phi\in \mathcal V$. Similarly to Propostion \ref{prop2}, we have the following result

\begin{proposition}\label{prop3}
Assume the hypothesis $H(\m j)$ is fulfilled. Then the functional $\m J$ is locally Lipschitz and there exists $\m c_{\m J}>0$ such that the following inequalities hold
\[
\m J^0(\phi;\psi)\leq \m c_{\m J}\left(1+\|\phi\|^{}_{\mathcal V}\right)\|\psi\|_{\mathcal V},\quad\forall \phi,\,\psi\in \mathcal V
\]and 
\[
\|\theta\|_{\mathcal V}\leq \m c_{\m J}\left(1+\|\phi\|^{}_{\mathcal V}\right),\quad \forall \theta\in\partial(\m J_{|\mathcal V})(\phi),\,\phi\in \mathcal V
\]
\end{proposition}

We introduce the function spaces
\[
\mathcal W_0=L^2(0,T; \m W_0),\qquad \mathcal M_0=\{\phi\in \mathcal W_0\,|\, \frac{\partial\phi}{\partial t}\in \mathcal W_0\}
\]where the time derivative $\frac{\partial\phi}{\partial t}$ is understood in the sense of vector-valued distributions.  The norm
\[
\|\phi\|_{\mathcal M_0}:=\|\phi\|_{\mathcal W_0}+\|\frac{\partial\phi}{\partial t}\|_{\mathcal W_0}
\]make the space $\mathcal M_0$ a Banach space. Moreover, the embeddings $\mathcal M_0\subset L^2(0,T;\ell^2(\m V,\mu))$ and $\mathcal M_0\subset C(0,T;\ell^2(\m V,\mu))$ are compact and continuous, respectively.

The purpose of this section is to prove the existence of solutions for the parabolic hemivariational inequalities on graphs, which can be stated as follows: Find $\phi\in\mathcal M_0$ such that
\begin{equation}\label{para0}
\begin{cases}
\,\phi'+\LG\phi+\partial\m J(\phi)\ni\m f&\qquad \text{in }\m V\times (0,T) \\
\,\phi(\m v,0)=\phi_0 &\qquad \text{in }\m V
\end{cases}
\end{equation}

\begin{definition}
We say that $\phi\in\mathcal M_0$ is a weak solution to problem \eqref{para0}, if $\phi(0,\m v)=\phi_0(\m v)$ in $\m V$, and the following inequality holds
\begin{align*}
\int_0^\m T\sum_{\m v\in\m V}\mu(\m v)\frac{\partial \phi(t,\m v)}{\partial t}\left(\psi(t,\m v)-\phi(t,\m v)\right)\,dt+\int_0^\m T\LG\phi(t,\m v)\left(\psi(t,\m v)-\phi(t,\m v)\right)\,dt\\
+\int_0^\m T \sum_{\m v\in \m V}\mu(\m v)\m j^0(t,\phi(t,\m v);\psi(t,\m v)-\phi(t,\m v))\,dt\geq \int_0^\m T\sum_{\m v\in \m V}\m f(t,\m v)\left(\psi(t,\m v)-\phi(t,\m v)\right)\, dt
\end{align*}for all $\psi\in\mathcal W_0$.
\end{definition}

The following theorem is the main result of this section. 

\begin{theorem}\label{evol}
Let $\m f\in\mathcal W_0$ and assume that the hypotheses $H(\m G)$ and $\m H(\m j)$ are fulfilled. If $\m\alpha_{\m J}<\frac{1}{2}\underline\alpha_\gamma\wedge\underline\alpha_\mu$, then the problem \eqref{para0} admits a weak solution.
\end{theorem}

\begin{proof}
First, we define the operator $\Lambda:\mathcal W_0\rightarrow\mathcal W_0$ by 
\[
\Lambda(\phi)(\psi):=\int_0^\m T\sum_{\m v\in\m V}\mu(\m v)\psi(t,\m v)\LG(\phi(t,\m v)+\tilde\phi_0(t,\m v))\,dt,
\]for all $\phi,\,\psi\in\mathcal W_0$ and where $\tilde\phi_0$ is such that $\tilde\phi_0(t,\m v)=\phi_0(\m v)$ for all $(t,\m v)\in (0,\m T)\times\m V$. From \eqref{ineqLG2} we have
\begin{align*}
\Lambda(\phi)(\psi)&=\int_0^\m T\sum_{\m v\in\m V}\mu(\m v)\psi(t,\m v)\LG(\phi(t,\m v)+\tilde\phi_0(t,\m v))\,dt\\
&\leq \frac{1}{2}\overline\alpha_\gamma\vee\overline\alpha_\mu\int_0^\m T\|\psi\|_{\m W_0}\,\|\phi+\tilde\phi_0\|_{\m W_0}\,dt\\
&\leq \frac{1}{2}\overline\alpha_\gamma\vee\overline\alpha_\mu\,\|\psi\|_{\mathcal W_0}\,\|\phi+\tilde\phi_0\|_{\mathcal W_0}
\end{align*}It follows that
\[
\|\Lambda(\phi)\|_{\mathcal W_0}\leq \frac{1}{2}\overline\alpha_\gamma\vee\overline\alpha_\mu \left(\|\phi\|_{\mathcal W_0}+\|\tilde\phi_0\|_{\mathcal W_0}\right),\quad\text{for all }\phi\in\mathcal W_0.
\]Thus, the operator $\Lambda$ is continuous. Moreover, the operator $\Lambda$ is strongly monotone. In fact, from \eqref{ineqLG1}, one can obtain 
\begin{align*}
\la\Lambda(\phi)-\Lambda(\psi),\phi-\psi\ra_{\mathcal W_0}&=\la\Lambda(\phi),\phi-\psi\ra_{\mathcal W_0}-\la\Lambda(\psi),\phi-\psi\ra_{\mathcal W_0}\\
&=\int_0^\m T\sum_{\m v\in\m V}\mu(\m v)(\phi(t,\m v)-\psi(t,\m v))\LG(\phi(t,\m v)+\tilde\phi_0(t,\m v))\,dt\\
&\quad-\int_0^\m T\sum_{\m v\in\m V}\mu(\m v)(\phi(t,\m v)-\psi(t,\m v))\LG(\psi(t,\m v)+\tilde\phi_0(t,\m v))\,dt\\
&=\int_0^\m T\sum_{\m v\in\m V}\mu(\m v)(\phi(t,\m v)-\psi(t,\m v))\LG(\phi(t,\m v)-\psi(t,\m v))\,dt\\
&\geq \frac{1}{2}\underline\alpha_\gamma\wedge\underline\alpha_\mu\int_0^\m T\|\phi(t,.)-\psi(t,.)\|^2_{\m W_0}\,dt\\
&=\frac{1}{2}\underline\alpha_\gamma\wedge\underline\alpha_\mu\,\|\phi-\psi\|^2_{\mathcal W_0}
\end{align*}for all $\phi,\,\psi\in\mathcal W_0$.

Define the operator $\m L:\m D(\m L)\subset\mathcal W_0\rightarrow \mathcal W_0$ by 
\[
\m L\phi=\frac{\partial\phi}{\partial t},\qquad\m D(\m L):=\{\phi\in\mathcal M_0\,|\,\phi(0)=0\}
\]which is closed, linear, densely defined and maximal monotone operator \cite{Z90}.

Now, we shall prove the hypotheses of the surjectivity theorem.
\vspace{0.10cm}

\textbf{Claim 1:} The multivalued operator $\Lambda+\partial\m J(.+\tilde \phi_0):\mathcal W_0\rightarrow 2^{\mathcal W_0}$ is bounded and pseudomonotone with respect to $\m D(\m L)$.

In fact, by the properties of Clarke's subdifferential, we deduce that the set $\Lambda(\phi)+\partial\m J(\phi+\tilde\phi_0)$ is nonempty, closed and convex in $\mathcal W_0$ for all $\phi\in\mathcal W_0$. By Proposition \ref{prop3} and the continuity of $\Lambda$, we obtain
\begin{align*}
\|\Lambda\phi+\xi\|_{\mathcal W_0}&\leq \|\Lambda\phi\|_{\mathcal W_0^*}+\|\xi\|_{\mathcal W_0^*}\\
&\leq \frac{1}{2}\overline\alpha_\gamma\vee\overline\alpha_\mu \left(\|\phi\|_{\mathcal W_0}+\|\tilde\phi_0\|_{\mathcal W_0}\right)+\alpha_\m J\left(1+\|\phi\|_\mathcal V+\|\tilde\phi_0\|_\mathcal V\right)\\
&\leq \alpha_\m J+(\alpha_\m J+\frac{1}{2}\overline\alpha_\gamma\vee\overline\alpha_\mu)(\|\phi\|_{\mathcal W_0}+\|\tilde\phi_0\|_{\mathcal W_0})
\end{align*}which implies that $\Lambda+\partial\m J(.+\tilde \phi_0):\mathcal W_0\rightarrow 2^{\mathcal W_0}$ is bounded. Moreover, since $\Lambda$ is linear and continuous (hence demicontinuous) and $\partial\m J$ is uppersemicontinuous from $\mathcal W_0$ to $w-\mathcal W_0$.

It remains to verify the last condition. let $\{\phi_n\}\subset\m D(\m L)$ and $\{\phi_n^*\}\subset\mathcal W_0$ be such that $\phi_n\rightarrow\phi$ weakly in $\mathcal W_0$, $\m L\phi_n\rightarrow\m L\phi$ weakly in $\mathcal W_0$, $\phi_n^*\in\Lambda\phi_n+\partial\m J(\phi_n+\tilde\phi_0)$ with $\phi_n^*\rightarrow\phi^*$ weakly in $\mathcal W_0$, and
\begin{equation}\label{eqq}
\limsup_{n\to\infty}\la\phi_n^*,\phi_n-\phi\ra_{\mathcal W_0}\leq 0
\end{equation}
Then, we are able to find a sequence $\{\xi_n\}\subset\mathcal W_0$ such that $\xi_n\in\partial\m J(\phi_n+\tilde\phi_0)$ and
\[
\phi_n^*=\Lambda\phi_n+\xi_n,\qquad\text{for each }n\in\mathbb N
\]
Consequently, from \eqref{eqq}, we get
\begin{equation}\label{eqsup}
\limsup_{n\to\infty}\la\Lambda\phi_n,\phi_n-\phi\ra_{\mathcal W_0}+\liminf_{n\to\infty}\la\xi_n,\phi_n-\phi\ra_{\mathcal W_0}\leq 0
\end{equation}
Since $\m W_0\subset \ell^2(\m V,\mu)$ and the embedding of $\m W_0$ in $\ell^2(\m V,\mu)$ is compact, we have that $\phi_n$ strongly converges to $\phi$ in $\mathcal V$. Furthermore, one has
\begin{equation}\label{inj}
\partial(\m J_{|\mathcal W_0})(\phi)\subset\partial(\m J_{|\mathcal V})(\phi)
\end{equation}which means that 
\begin{equation}\label{eqqq}
\la\xi_n,\phi_n-\phi\ra_{\mathcal W_0}=\la\xi_n,\phi_n-\phi\ra_\mathcal V
\end{equation}
Further, from the boundedness of $\{u_n\}$ in $\mathcal W_0$, we have that $\{\xi_n\}$ is bounded both in $\mathcal V$ and in $\mathcal W_0$. Then,  by \eqref{eqqq}, we pass to the limit as $n\to\infty$ to get
\[
\lim_{n\to\infty}\la\xi_n,\phi_n-\phi\ra_{\mathcal W_0}=\lim_{n\to\infty}\la\xi_n,\phi_n-\phi\ra_\mathcal V=0
\]This convergence combined with \eqref{eqsup} and the monotonicity of $\Lambda$ implies
\[
\limsup_{n\to\infty}\|\phi_n-\phi\|_{\mathcal W_0}^2\leq \m A^{-1}\limsup_{n\to\infty}\la\Lambda\phi_n,\phi_n-\phi\ra_{\mathcal W_0}+\m A^{-1}\lim_{n\to\infty}\la\Lambda\phi,\phi-\phi_n\ra_{\mathcal W_0}\leq 0
\]where $\m A=\frac{1}{2}\underline\alpha_\gamma\wedge\underline\alpha_\mu$. Hence $\phi_n\to\phi$ strongly in $\mathcal W_0$. On the other side, the reflexivity of $\mathcal W_0$ and boundedness of $\{\xi_n\}\subset\mathcal W_0$ allow to assume, at least for a subsequence, that $\xi_n$ converges weakly in $\mathcal W_0$ to some $\xi\in\mathcal W_0$.
Since $\partial\m J$ is upper semicontinuous from $\mathcal W_0$ to $w-\mathcal W_0$ and it has convex and closed values, it is closed from $\mathcal W_0$ to $w-\mathcal W_0$ (see \cite[Theorem 1.1.4]{KOZ01}). Therefore, we obtain $\xi\in\partial\m J(\phi+\tilde\phi_0)$.

To conclude, we have $\phi^*=\xi+\Lambda\phi\in\Lambda\phi+\partial\m J(\phi+\tilde\phi_0)$ and
\[
\la \phi_n^*,\phi_n\ra_{\mathcal W_0}=\la\xi_n+\Lambda\phi_n,\phi_n\ra_{\mathcal W_0}\to \la\xi+\Lambda\phi,\phi\ra_{\mathcal W_0}=\la \phi^*,\phi\ra_{\mathcal W_0}
\]which means that the operator $\Lambda+\partial\m J(.+\tilde \phi_0):\mathcal W_0\rightarrow 2^{\mathcal W_0}$ is pseudomonotone with respect to $\m D(\m L)$.

\textbf{Claim 2:} The operator $\Lambda+\partial\m J(.+\tilde \phi_0):\mathcal W_0\rightarrow 2^{\mathcal W_0}$ is coercive.

For all $\phi\in\mathcal W_0$ one has
\begin{align*}
&\la \Lambda\phi+\partial\m J(\phi+\tilde\phi_0),\phi\ra_{\mathcal W_0}=\la \Lambda\phi,\phi\ra_{\mathcal W_0}+\la \partial\m J(\phi+\tilde\phi_0),\phi\ra_{\mathcal W_0}\\
&\qquad\geq\m A\,\|\phi\|^2_{\mathcal W_0}-\m A\,\|\tilde\phi_0\|_{\mathcal W_0}\|\phi\|_{\mathcal W_0}+\la \partial\m J(\phi+\tilde\phi_0),\phi\ra_{\mathcal V}\\
&\qquad\geq\m  A\,\|\phi\|^2_{\mathcal W_0}-\m A\,\|\tilde\phi_0\|_{\mathcal W_0}\|\phi\|_{\mathcal W_0}-\| \partial\m J(\phi+\tilde\phi_0)\|_{\mathcal V}\,\|\phi\|_{\mathcal V}\\
&\qquad\geq\m  A\,\|\phi\|^2_{\mathcal W_0}-\m A\,\|\tilde\phi_0\|_{\mathcal W_0}\|\phi\|_{\mathcal W_0}-\alpha_\m J\left(1+\|\phi+\tilde\phi_0\|_{\mathcal V}\right)\|\phi\|_{\mathcal V}\\
&\qquad\geq\m  A\,\|\phi\|^2_{\mathcal W_0}-\m A\,\|\tilde\phi_0\|_{\mathcal W_0}\|\phi\|_{\mathcal W_0}-\alpha_\m J\|\phi\|_{\mathcal W_0}-\alpha_{\m J}\|\phi\|^2_{\mathcal W_0}-\alpha_{\m J}\|\phi\|_{\mathcal W_0}\|\tilde\phi_0\|_{\mathcal W_0}\\
&\qquad\geq \left(\left(\m A-\alpha_\m J\right)\|\phi\|_{\mathcal W_0}-(\m A+\alpha_\m J)\,\|\tilde\phi_0\|_{\mathcal W_0}+\alpha_\m J\right)\|\phi\|_{\mathcal W_0}\\
&\qquad\geq c(\|\phi\|_{\mathcal W_0})\|\phi\|_{\mathcal W_0}
\end{align*}
where $c:\mathbb R^+\rightarrow\mathbb R$ with $c(r)=\left(\m A-\alpha_\m J\right)r-(\m A+\alpha_\m J)\,\|\tilde\phi_0\|_{\mathcal W_0}+\alpha_\m J$. It is clear that $c(r)\rightarrow\infty$ as $r\to\infty$, thus the operator $\Lambda+\partial\m J(.+\tilde \phi_0):\mathcal W_0\rightarrow 2^{\mathcal W_0}$ is coercive.

We are now in a position to apply the surjectivity result. We deduce that there exists a function $\chi\in\mathcal W_0$ with $\chi(0)=0$ solving the following inclusion problem
\begin{equation}\label{prob00}
\begin{cases}
\m L\chi+\Lambda\chi+\partial\m J(\chi+\tilde\phi_0)\ni\m f,\quad\text{ in }\mathcal W_0\\
\chi(0)=0.
\end{cases}
\end{equation}
\textbf{Claim 3:} If $\chi\in\mathcal M_0$ is a solution to problem \eqref{prob00}, then $\phi=\chi+\tilde\phi_0$ is a weak solution to problem \eqref{para0}.

Let $\chi\in\mathcal M_0$ be a solution to problem \eqref{prob00}. Hence $\phi=\chi+\tilde\phi_0$ solves the following problem

\begin{equation}
\begin{cases}
\frac{\partial\phi}{\partial t}+\Lambda(\phi-\tilde\phi_0)+\partial\m J(\phi)\ni\m f,\quad\text{ in }\mathcal W_0\\
\phi(0)=\phi_0.
\end{cases}
\end{equation}
By the definition of generalized Clarke subdifferential we obtain $\eqref{para0}$. This completes the proof.

\end{proof}

\section{Concluding remarks}

In this section we give some remarks and extensions of the results proved in previous sections.

\begin{enumerate}
\item Let $\Phi:\m W_0\rightarrow\bar{\mathbb R}$ be a proper, convex and lower semicontinuous functional such that $0\in\partial_C\Phi(\phi_0)$, where $\partial_C$ is the subdifferential in the sense of convex analysis. Suppose additionally that $\phi_0\in\mathrm{int }D(\Phi)$. Then, the variational-hemivariational inequality: Find $\phi\in\mathcal M_0$ such that 
\begin{equation}\label{para}
\begin{cases}
\,\phi'+\LG\phi+\partial\m J(\phi)+\partial_C\Phi(\phi)\ni\m f&\qquad \text{in }\m V\times (0,T) \\
\,\phi(\m v,0)=\phi_0 &\qquad \text{in }\m V
\end{cases}
\end{equation}admits a solution. To prove the existence of Problem \eqref{para}, let us consider the functional $\Psi:\mathcal W_0\rightarrow\bar{\mathbb R}$ defined by
\[
\Psi(\phi)=\int_0^\m T\Phi(\phi+\tilde\phi_0)\,dt
\]Now, it suffices to continue on the proof of Theorem \ref{evol} and prove, additionally, that the operator $\partial_C\Psi$ is maximal monotone and strongly quasi-bounded with $0\in\partial\Psi(0)$. The existence follows by the surjectivity result stated by Theorem 3.1 in \cite{GMO15}. 
\item One can think about an alternative proof of existence in both the elliptic and parabolic problems by using the Galerkin scheme adapted to graph theory context. Let $(\m G_n)_{n\geq 0}$ be a growing family of finite graphs that exhaust $\m G$ in the sense of \cite[Definition 4.1]{M82} and \cite[Definition 3.3]{M13} and consider the problem of finding $\phi_n$ such that 
\begin{equation}\label{approx}
\phi_n'+\LGn\phi_n+\sum_{\m v\in\m V_n}\mu(\m v)\m j_n'(\phi_n(\m v))=\m f
\end{equation}where $\m j_n$ is a mollification of $\m j$. By using techniques from the proof of Theorem 3.6 in \cite{HM15} and some standard calculation on the nonlinear term, one can prove that the sequence $(\phi_n)_n$ on $\m G_n$ is bounded in $H^1(0,\m T;\ell^2(\m V,\mu))$ and weak${}^*$ in $L^\infty(0,\m T;\m W_0)$. By taking a subsequence if necessary, it possible to prove that $(\phi_n)_n$ converges to some $\phi$ and $\m j_n'(\phi_n)$ converges in $\ell^2(\m V,\mu)$ to some $\xi$. By the convergence theorem of Aubin and Cellina \cite{AC84}, it is clear that $\xi\in\partial\m j(\phi)$ and by taking the limit in \eqref{approx}, one can see that $\phi$ resolves Problem \eqref{evol}.

\item Let $\m h$ be some nonnegative continuous function which satisfies with $\m j$ the following growth condition
\[
|\m h(\xi_1)\xi|\leq c(1+|\xi_1|+|\xi_2|),\quad \text{for all }\xi_1,\,\xi\in\mathbb R,\text{ with } \xi\in\partial\m j(\xi_2)
\]where $c$ is some nonnegative constant. One can prove a version of Aubin-Clarke theorem for discrete functionals in the form:
\[
\m J(\phi)=\sum_{\m v\in \m V}\mu(\m v)\m h(\phi(\m v))\m j(\phi(\m v))
\]With the above hypotheses we have for $\phi$, $\psi$ in $\ell^2(\m V,\mu)$ that
\[
\m J^\circ(\phi;\psi)\leq \sum_{\m v\in\m V}\mu(\m v)\m h(\m v)\m j^\circ(\phi(\m v);\psi(\m v))
\]
With some modifications, one can prove that the quasi-hemivariational versions of Problems \eqref{para0} and \eqref{evol} admit weak solutions.

\item By using Theorem 3.6 in \cite{HM15}, the theory in this paper can be applied for the operator $\mathscr L_{\gamma,0}^\m G$ if we assume that $\m G$ is uniformly locally finite and satisfy the $d-$ isoperimetric inequality for some $d\geq 2$
\end{enumerate}

\par\bigskip

\subsection*{Acknowledgments}

\end{document}